\documentclass{article}
\usepackage{makeidx}
\makeindex
\usepackage{amsmath}
\usepackage{amsfonts}
\usepackage{amssymb}
\usepackage{amsthm}
\theoremstyle{plain}
\newtheorem{theorem}{Theorem}[section]
\newtheorem{lemma}[theorem]{Lemma}

\newtheorem{corollary}[theorem]{Corollary}
\theoremstyle{definition}

\theoremstyle{remark}
\newtheorem*{remark}{Remark}

\begin{document}

\title{New Constants in Discrete Lieb-Thirring Inequalities for Jacobi Matrices}
\author{Arman Sahovic}
\date{\today}
\maketitle

\begin{abstract}

\noindent
This paper is essentially derived from the observation that some results used for 
improving constants in the Lieb-Thirring inequalities for Schr\"odinger operators in $L^2(-\infty,\infty)$ can be translated to the discrete Schr\"odinger operators and more generally to Jacobi matrices. 
Some results were previously proved  by D. Hundertmark and B. Simon and the aim of this article is to improve the constants obtained in their article \cite{Hun}.
\end{abstract}


\section{Introduction and main results}


Let $W$ be a tridiagonal self-adjoint Jacobi matrix
\begin{equation*}
W = \left(
\begin{array}{cccccc}
\ddots & \vdots & \vdots & \vdots & \vdots & \ddots\\
\ldots & b_{-1} & a_{-1} & 0   & 0   & \ldots \\
\ldots & a_{-1} & b_0 & a_0 & 0   & \ldots \\
\ldots & 0   & a_0 & b_1 & a_1   & \ldots \\
\ldots & 0   & 0   & a_1 & b_2 & \ldots \\
\ddots & \vdots & \vdots & \vdots & \vdots & \ddots
\end{array} \right)
\end{equation*}
viewed as an operator in $\ell^2(\Bbb{Z})$  of complex sequences: 
\begin{equation}\label{W}
Wu(n)=a_{n-1}u(n-1)+b_nu(n)+a_nu(n+1), \qquad n\in\mathbb Z.
\end{equation}
In what follows we assume that $a_n>0$, $b\in\mathbb R$ and $a_n\to 1$, $b_n\to 0$, $n\to\pm\infty$, rapidly enough. Then the essential  spectrum $\sigma_{ess} = [-2,2]$ and $W$ may have simple eigenvalues $\{E_j^\pm\}_{j=1}^{N_\pm}$ where $N_\pm\in \mathbb N \cup \{\infty\}$,
$$
E_1^{+}>E_2^{+}>...>2>-2>...>E_2^{-}>E_1^{-}.
$$

\noindent Applying the method of D. Hundertmark, E. Lieb and L.Thomas  \cite{Tho} to Jacobi matrices,  D.Hundertmark and B.Simon  \cite{Hun} have proved the following result concerning Lieb-Thirring inequalities for Jacobi matrices:
\begin{theorem}[Hundertmark-Simon]\label{Hun-Sim}
Let $\{b_n\}, \, \{a_n-1\}\in l^{1}$. Then 
\begin{equation}\label{HS1}
\sum_{j=1}^{N_+} \big((E_j^+)^2 -4\big)^{1/2} + \sum_{j=1}^{N_-} \big((E_j^-)^2 -4\big)^{1/2}
\le \sum_{n=-\infty}^{\infty} |b_n| + 4\sum_{n=-\infty}^{\infty}|a_n-1|.
\end{equation}
Moreover, if $\{b_n\}, \, \{a_n-1\}\in l^{\gamma+1/2}$, $\gamma\geq1/2$, then
\begin{equation}\label{HS2}
\sum_{j=1}^{N_+}|E^+_j-2|^\gamma+|E^-_j+2|^\gamma \leq c_\gamma\left[\sum_{n=-\infty}^{\infty} |b_n|^{\gamma+1/2}
+4\sum_{n=-\infty}^{\infty}|a_n-1|^{\gamma+1/2}\right]
\end{equation}
where

$$
c_\gamma=2(3^{\gamma-1/2})L^{cl}_{\gamma,1}.
$$
where
$$
L^{cl}_{\gamma,1}=\frac{\Gamma(\gamma+1)}{2\sqrt{\pi}\,\,\Gamma(\gamma+3/2)}
$$

\end{theorem}

\begin{remark}
Note that the constants in front of each sum in the right hand side of \eqref{HS1} are sharp and this is the only case when sharp constants in Lieb-Thirring inequalities are known for Jacobi matrices. 
\end{remark}

This paper is one of a series of papers where we would like to obtain improved and possibly sharp constants for discrete operators. 
In particular, the aim of this paper is to improve the constants $c_\gamma$ appearing in \eqref{HS2} for $\gamma\ge1$ by applying the  method of A. Eden and C. Foias, \cite{Ede} who obtained improvement of constants in Lieb-Thirring inequalities for one-dimensional Schr\"odinger operators acting in $L^2(\Bbb R)$. 

Let us first recall known results for  ``continuous" multi-dimensional Scr\"odinger operators that give us a motivation for the study of discrete problems.
Let $H$ be a Schr\"odinger operator acting in  $L^2(\mathbb R^d)$ 
        \begin{equation}\label{Schrodinger}
       H \psi_j:=  -\Delta \psi_j(x) + V\psi_j(x) = -e_j \psi_j(x).  
        \end{equation}
Lieb-Thirring inequalities relate the eigenvalues $\{e_j\}$ of the operator $H$ and the potenvial $V\in L^{\gamma+n/2}(\Bbb{R}^d)$ via the following estimate
        \begin{equation}\label{LTI}
        \sum_j|e_j|^{\gamma}  \leq  L_{\gamma,d}\int V_-(x)^{\gamma+d/2}d x,
        \end{equation} 
where $V_- = (|V|-V)/2$ is the negative part of $V$.

\noindent   
It is known that the constants $L_{\gamma,d}$ are finite if $\gamma\geq1/2$ $(d=1)$, 
$\gamma>0$ ($d=2)$, and $\gamma\geq0 $ ($d\geq3$).
If $\gamma=0$ ($d\geq 3$) the inequality \eqref{LTI} is called the CLR-inequality  (Cwikel-Lieb-Rozenblum). The case $\gamma=1/2$ ($d=1$) was justified by T. Weidl in \cite{W}. 
In all these cases we have the following Weyl asymptotic formula for the eigenvalues of the operator 
$H(\alpha) =  -\Delta + \alpha V$ 
\begin{multline}\label{asymp}
 \sum_j|e_j|^{\gamma}  = \alpha^{\gamma + d/2} \, (2\pi)^{-d}\, \int\int (|\xi|^2 + V(x))_-^\gamma
 dx d\xi  + o\,(\lambda^{\gamma + d/2})\\
 =  \alpha^{\gamma + d/2} L_{\gamma,d}^{cl} \int V_-(x)^{\gamma+d/2}d x + 
 o\,(\lambda^{\gamma + d/2}),
 \quad {\mathrm as} \quad \alpha \to \infty,
\end{multline}
where 
$$
L_{\gamma,d}^{cl} = (2\pi)^{-d}\,\int (|\xi|^2 -1)_-^\gamma d\xi.
$$
Therefore the sharpness of the constants  $L_{\gamma,d}$ appearing in \eqref{LTI} could be compared with the values of $L_{\gamma,d}^{cl}$. Clearly \eqref{asymp} implies that 
$L_{\gamma,d}\leq L_{\gamma,d}^{cl}$.

\noindent
In some cases the values of sharp constants $L_{\gamma,d}$ are known. However, they do not always coincide with  $L_{\gamma,d}^{cl}$.
It has been proved in \cite{LTh} that  $L_{3/2,1} = 3/16$ and by using \cite{AL} one obtains sharp 
constants $L_{\gamma,1}$ for all $\gamma\geq 3/2$. Later A.  Laptev and T. Weidl \cite{LaWe} obtained sharp constants for $L_{\gamma,d}$ for all $\gamma\geq 3/2$ in any dimension.  
If $\gamma=1/2$ and $d=1$  then  $L_{1/2,1} = 1/2 $ was found by D. Hundertmark, E.B. Lieb and L. Thomas in \cite{Tho}. 

Several attempts have been made to improve estimates for the constants  $L_{\gamma,d}$.
For $1/2\leq\gamma<3/2$, Hundertmark, Laptev and Weidl \cite{Wei} found the constant to
be not greater than $2 \,L_{\gamma,d}^{cl}$. Recently  this has been improved for $1\leq \gamma\leq 3/2$ by J. Dolbeault, A. Laptev and  M. Loss \cite{Dol} to $c\,L_{\gamma,d}^{cl}$, $c=1.8...$, using methods essentially derived from Eden and Foias \cite{Ede}, which we will use ourselves in quite a substantial way.
Eden and Foias essentially used an interesting method to improve the constant in the Lieb-Thirring inequalities in one dimension.


\noindent
Let us now introduce the operator $D$ in $l^2(\mathbb Z)$
\begin{equation} \label{ddo}
D {\varphi(n)} = {\varphi(n+1)-\varphi(n)}, \qquad n\in\mathbb Z.
\end{equation}
\noindent 
We choose its adjoint to be:
$$
D^* \varphi(n) =-( \varphi(n)-\varphi(n-1)).
$$
\noindent Then the discrete Laplacian takes the form:
$$
D^*D\,\varphi(n) = -\varphi(n+1)-\varphi(n-1)+2\varphi(n).
$$
The spectrum $\sigma(D^*D)$ of this operator  is absolutely continuous. We have
$\sigma(D^*D) = [0,4]$ and if $b_n\ge 0$ then the discrete Schr\"odinger operator 
\begin{equation}\label{Sch}
H_D:=D^*D  - b_n
\end{equation}
may have negative eigenvalues.

\noindent
Our first result is:

%
\begin{theorem}\label{Lieb-Thirring}
Let $b_n\geq 0$, $\{b_n\}_{n=-\infty}^\infty \in l^{3/2}(\mathbb Z)$. Then the
negative eigenvalues $\{e_j\}$ of the operator in \eqref{Sch} are discrete
and they satisfy the inequality 
\begin{equation}\label{eq:lieb}
\sum_{j}\,|e_j |\le\, \frac{\pi}{\sqrt3}\,\,L^{cl}_{1,1}\sum_{n\in\Bbb{Z}}\, b_n^{3/2}.
\end{equation}
\end{theorem}


\noindent
The standard argument due to Aizenman and Lieb \cite{AL} implies that we obtain the following
spectral inequalities for moments $\gamma\geq 1$. 


\begin{theorem}\label{AisLieb}
Let $b_n\geq 0$, $\{b_n\}_{n=-\infty}^\infty \in l^{\gamma+1/2}(\mathbb Z)$, $\gamma \ge 1$. Then the negative eigenvalues  $\{e_j\}$ of the operator in \eqref{Sch}  satisfy the inequality 
\begin{equation}\label{eq:lieb}
\sum_{j}\,|e_j |^\gamma \le\,C \,\,\sum_{n\in\Bbb{Z}}\, b_n^{\gamma + 1/2},
\end{equation}
where as in the continuous case
$$
C = \frac{\pi}{\sqrt 3}\, L_{\gamma,1}^{cl}.
$$
\end{theorem}


\noindent
By changing sign of the operator we immediately obtain a version of Theorem \ref{AisLieb} for positive eigenvalues of the operator $-D^*D +b_n$, $b_n\geq 0$. 


\begin{corollary}\label{AisLieb+}
Let $b_n\geq 0$, $\{b_n\}_{n=-\infty}^\infty \in l^{\gamma+1/2}(\mathbb Z)$, $\gamma \ge 1$. Then the positive eigenvalues  $\{e_j\}$ of the operator $-D^*D +b_n$  satisfy the inequality 
\begin{equation}\label{eq:lieb+}
\sum_{j}\,|e_j |^\gamma \le\,C \,\,\sum_{n\in\Bbb{Z}}\, b_n^{\gamma + 1/2},
\end{equation}
where as in the continuous case 
$$
C = \frac{\pi}{\sqrt 3}\, L_{\gamma,1}^{cl}.
$$
\end{corollary}

\noindent
Our main result for Jacobi matrices is the following statement:

\begin{theorem}\label{main}
Let $\gamma\geq1,\,\,\{b_n\}_{n=-\infty}^\infty, \, \{a_n-1\}_{n=-\infty}^\infty \in l^{\gamma+1/2}$.  Then for the eigenvalues of the operator \eqref{W} we have 
\begin{equation}\label{AS}
\sum_{j}|E^{-}_{j}+2|^\gamma+|E^{+}_{j}-2|^\gamma \leq 3^{\gamma-1/2}\frac{\pi}{\sqrt{3}}\,L^{cl}_{\gamma,1}
\left(\sum_n |b_n|^{\gamma+1/2}+4\sum_n |a_n-1|^{\gamma+1/2})\right).
\end{equation}
\end{theorem}

\begin{remark}
Comparing the constants in right hand sides of \eqref{HS2} and \eqref{AS} we note that the latter constant is around 1.1 times better.

\end{remark}


\section{Some auxiliary results}


\noindent
In order to prove our main results we consider some auxiliary statements. 

\begin{lemma}\label{agmon}
Let $\varphi,\, D\varphi \in l^2(\mathbb{Z})$. Then for any $n\in \mathbb Z$ 
        $$
        \mid\varphi(n)\mid^2  \leq  \| \varphi \|_{\ell^2} \| D\varphi\|_{\ell^2}
        $$
i.e.    $$
        \mid\varphi(n)\mid^2  \leq  \left(\sum_{k=-\infty}^\infty \mid\varphi(k)\mid^2\right)^{1/2}
        \left(\sum_{k=-\infty}^\infty\mid D\varphi(k)\mid^2\right)^{1/2}.
        $$
\end{lemma}

 
\begin{proof}
For any $n\in\mathbb Z$ we have
\begin{eqnarray*}
       |\varphi(n)|^2 & = & 1/2\left|\sum_{k=-\infty}^{n}D(\varphi^2(k))-\sum_{k=n+1}^{\infty}D(\varphi^2(k))\right|\\
                      & \leq & 1/2 \Big(\sum_{k=-\infty}^{n}\left|D(\varphi^2(k))\right|+\sum_{k=n+1}^{\infty}\left|D(\varphi^2(k))\right| \Big)\\
                      & = & 1/2\sum_{k=-\infty}^{\infty}\left|\varphi^2(k+1)-\varphi^2(k)\right|\\
                      & = & 1/2\sum_{k=-\infty}^{\infty}|D\varphi(k)|\Bigl(|\varphi(k+1)|+
                            |\varphi(k)|\Bigr).
\end{eqnarray*}
Now we apply the Cauchy-Schwarz inequality and obtain
\begin{multline*}             
  |\varphi(n)|^2 \leq 1/2\sum_{k=-\infty}^{\infty}|D\varphi(k)|\Bigl(|\varphi(k+1)|+|\varphi(k)|\Bigr)\\
       \leq  1/2\left(\sum_{k=-\infty}^{\infty}|D\varphi(k)|^2\right)^{1/2}
                \left[\left(\sum_{k=-\infty}^{\infty}|\varphi(k+1)|^2\right)^{1/2}+
                \left(\sum_{k=-\infty}^{\infty}|\varphi(k)|^2\right)^{1/2}\right]\\
       =  \Bigl(\sum_{k=-\infty}^\infty|D\varphi(k)|^2\Bigr)^{1/2}\Bigl(\sum_{k=-\infty}^\infty |\varphi(k)|^2\Bigr)^{1/2}.
\end{multline*}
The proof is complete.
\end{proof}

\noindent
The next result is a discrete version of the result obtained by Eden and Foias in \cite{Ede}.

%
\begin{lemma}[Discrete Generalised Sobolev Inequality]\label{DGSI}
Let $\{\psi_j\}_{j=1}^N$ be an orthonormal system of function in $l^2(\mathbb Z)$
and let $\rho(n) = \sum_{j=1}^N \psi^2_j(n)$.
Then
\begin{equation*}
\sum_{n\in\Bbb Z} \rho^3(n)\, = \sum_{n\in\Bbb{Z}} \Bigl(\sum_{j=1}^N |\psi_j(n)|^2\Bigr)^3 \,
 \le \sum_{j=1}^N \sum_{n\in\Bbb{Z}} |{D\psi}_j(n)|^2\, .
\end{equation*}
\end{lemma}

\begin{proof}
Let $\xi=(\xi_1, \xi_2, \dots, \xi_N)\in {\Bbb C}^N$. Then
by Lemma \ref{agmon} we obtain that for every $n\in\mathbb{Z}$:
        \begin{eqnarray*}
        \Big|\sum_{j=1}^N \xi_j \psi_j(n)\Big| & \le & \Bigl(\sum_{j,k=1}^N \xi_j\bar\xi_k\langle\psi_j,\psi_k\rangle\Bigr)^{1/4}\Bigl(\sum_{j,k=1}^N                                                        \xi_j\bar\xi_k \langle D\psi_j,D\psi_k\rangle\Bigr)^{1/4}\\
                                               & \le & \Bigl(\sum_{j=1}^N \xi_j^2\Bigr)^{1/4}\Bigl(\sum_{j,k=1}^n                                                        \xi_j\bar\xi_k \langle D\psi_j,D\psi_k\rangle\Bigr)^{1/4}.
        \end{eqnarray*}

\noindent 
If we set $\xi_j = \psi_j(n)$ then the latter inequality becomes
$$
\rho(n) = \sum_{j=1}^N |\psi_j(n)|^2\le \rho^{1/4}(n)
\Bigl(\sum_{j,k=1}^N \psi_j(n)\overline{\psi_k(n)} \langle D\psi_j,D\psi_k\rangle\Bigr)^{1/4}.
$$
Thus 
$$
\rho^3(n)\le \sum_{j,k=1}^N \psi_j(n)\overline{\psi_k(n)} \langle D\psi_j,D\psi_k\rangle.
$$
If we sum both sides also taking into account that $\{\psi_j(n)\}$ is an orthonormal system, then
we arrive at 
$$
\sum_{n\in\Bbb{Z}}\Bigl(\sum_{j=1}^N |\psi_j(n)|^2\Bigr)^3\,\le 
\sum_{j=1}^N \Bigl(\sum_{n\in\Bbb{Z}} |D\psi_j(n)|^2 \,\Bigr).
$$
\end{proof}


\begin{lemma}\label{UnitEq}
Discrete Schr\"odinger operators 
$$
-D^*D + b_n \quad {\mathrm and} \quad D^*D -4 + b_n
$$
are unitary equivalent.
\end{lemma}

\begin{proof}
Indeed, let $\mathcal F$ be the Fourier transform
$$
\mathcal F\varphi(\theta) = \hat\varphi(\theta) = \sum_{n=-\infty}^\infty \varphi(n) e^{in\theta}, 
\qquad \theta\in (0,2\pi).
$$
Then
\begin{equation*} 
\mathcal F(D^*D -4 + b_n)\mathcal F^* \varphi(\theta) =
2(1-\cos\theta) - 4\big) \hat\varphi(\theta)  + \int_0^{2\pi} \hat{b}(\theta -\tau) 
\hat\varphi(\tau) \,d\tau.
\end{equation*}
Therefore using periodicity of $\hat{b}$ and denoting  $\hat\psi(\theta) = \hat\varphi(\theta + \pi)$ 
 we find
\begin{multline*} 
\int_0^{2\pi} (-2-\cos\theta) |\hat\varphi(\theta)|^2 + \int_0^{2\pi}\int_0^{2\pi} \hat{b}(\theta - \tau)
\hat\varphi(\tau)\overline{\hat\varphi(\theta)}\, d\tau d\theta\\
= \int_0^{2\pi} (-2+\cos\theta) |\hat\psi(\theta)|^2 + \int_0^{2\pi}\int_0^{2\pi} \hat{b}(\theta - \tau)
\hat\psi(\tau)\overline{\hat\psi(\theta)}\, d\tau d\theta.
\end{multline*}
Since $\mathcal F(-D^*D)\mathcal F^*$ is unitary equivalent to  $-2+\cos\theta$ and this proves the  lemma.
\end{proof}


\noindent
\section{ Proof of Theorems \ref{Lieb-Thirring} and \ref{AisLieb}}

\noindent
We begin with the proof of Theorem \ref{Lieb-Thirring}.
\begin{proof}
\noindent
Let $\{\psi_j\}_{j=1}^N$ be the orthonormal system 
of eigenvectors corresponding to the negative eigenvalues $\{-e_j\}$
of the discrete Schr\"odinger operator:
$$
D^*D\psi_j - b_n\psi_j = -e_j \psi_j,
$$
where we assume that $b_n\geq0$ $\forall$ $n\in\Bbb{N}$.
\\
Then by using the latter result and H\"older's inequality we obtain just like before

$$
\sum_{n\in\Bbb{Z}}\Bigl(\sum_{j=1}^N |\psi_j|^2\Bigr)^3\, 
-\Bigl(\sum_{n\in\Bbb{Z}} b_n^{3/2}\, \Bigr)^{2/3} 
\Bigl(\sum_{n\in\Bbb{Z}}\Bigl(\sum_{j=1}^N |\psi_j|^2\Bigr)^3\, \Bigr)^{1/3}
$$
$$
\le \sum_{j=1}^N\Bigl(\sum_{n\in\Bbb{Z}}\Bigl(|D\psi_j|^2 - b_n|\psi_j|^2\Bigr)\,\Bigr) 
= -\sum_{j=1}^Ne_j.
$$
\noindent Denote 

$$
X= \Bigl(\sum_{n\in\Bbb{Z}}\Bigl(\sum_{j=1}^N |\psi_j|^2\Bigr)^3\,\Bigr)^{1/3},
$$
then the latter inequality can be written as

$$
X^3 - \Bigl(\sum_{n\in\Bbb{Z}} b_n^{3/2}\, \Bigr)^{2/3} X \le -\sum_{j=1}^Ne_j \label{dlt1}.
$$
Maximising the left hand side we find 

$$
X = \frac{1}{\sqrt3}\Bigl(\sum b_n^{3/2}\,\Bigr)^{1/3}.
$$
We substitute this back into \eqref{dlt1}:
\begin{eqnarray*}
        -\sum_{j=1}^Ne_j & \geq & \frac{1}{3\sqrt3}\sum_{n\in\Bbb{Z}} b_n^{3/2}\,-\frac{1}{\sqrt3}\sum_{n\in\Bbb{Z}}b_n^{3/2}\\
                         &   =  & -\frac{2}{3\sqrt3}\sum_{n\in\Bbb{Z}} b_n^{3/2}
\end{eqnarray*}

\noindent and we finally obtain the discrete version of the Lieb-Thirring Inequality:
\begin{equation}\label{eq:lieb}
\sum_{j=1}^N\,|e_j|\le\, \frac{2}{3\sqrt3}\,\,\sum_{n\in\Bbb{Z}}\, b_n^{3/2}
\end{equation}
The proof is complete.
\end{proof}

\noindent
Let $\mathcal B$ be a Beta-function
$$
\mathcal B(x,y) = \int_0^1 t^{x-1} (1-t)^{y-1}\, dt.
$$

\noindent
{\it Proof of Theorem \ref{AisLieb}}

\noindent
Let $\mathcal B$ be a Beta-function
$$
\mathcal B(x,y) = \int_0^1 \tau^{x-1} (1-\tau)^{y-1}\, d\tau.
$$
Then  by scaling we obtain that for any $\gamma>1$ and $\mu\in\mathbb R$
$$
\mu^\gamma_+= \mathcal B^{-1}(\gamma-1,2)\,  \int_0^\infty \, \tau^{\gamma-2} (\mu-\tau)_+ d\tau.
$$
Let $e_j(\tau)$ be the eigenvalues of the operator $D^*D -(b_n - \tau)_+$. Then by variational principle
for the negative eigenvalues $-(e_j -\tau)_+$ of the operator $D^*D -b_n + \tau$ we have  
$$
(e_j - \tau)_+ \leq e_j(\tau).
$$
Therefore for any $\gamma> 1$ applying Theorem \ref{Lieb-Thirring} to the the operator 
$D^*D -(b_n - \tau)_+$ we find
\begin{multline*}
\sum_j e_j^\gamma = \mathcal B^{-1}(\gamma-1,2)\, \int_0^\infty \sum_j(e_j-\tau)_+ \, \tau^{\gamma-2} d\tau\\
\leq \mathcal B^{-1}(\gamma-1,2)\, \int_0^\infty \sum_j e_j(\tau)_+ \, \tau^{\gamma-2} d\tau\\
\leq    \frac{2}{3\sqrt3}\,\mathcal B^{-1}(\gamma-1,2)\, \int_0^\infty \sum_n   \tau^{\gamma-2} 
(b_n - \tau)_+^{3/2}d\tau\\
=   \frac{2}{3\sqrt3}\,\mathcal B^{-1}(\gamma-1,2)\, \mathcal B(\gamma-1, 5/2) \, \sum_n b_n^{\gamma + 1/2}\\
=\frac{\pi}{\sqrt{3}}L^{cl}_{\gamma,1}\sum_n b_n^{\gamma + 1/2}
\end{multline*}
which proves the required statement.


\section{Proof of Theorem \ref{main} }


Let us introduce notations
$$
(b_n)_+=\max (b_n, 0), \qquad  (b_n)_-= - \min (b_n, 0)
$$
and let
$$
W(\{a_n\},\{b_n\})u(n) = Wu (n) =a_{n-1}u(n-1)+b_nu(n)+a_nu(n+1).
$$
Note that using this notation and Lemma \ref{UnitEq} we have
\begin{equation}\label{W=D1}
D^*D + b_n = W(\{a_n\equiv1\},\{b_n +2 \}), 
\end{equation}
\begin{equation}\label{W=D2}
-D^*D + b_n = W(\{a_n\equiv -1\},\{b_n - 2\}).
\end{equation}
Hundertmark and Simon \cite{Hun} made an observation, which allows us to use bounds for
$a_n\equiv1$ for the general case. Namely for any $a_n\in \mathbb R$
\begin{equation}\label{impbound}
\left(
        \begin{array}{cc}
        -|a_n-1| & 1        \\
        1        & -|a_n-1| 
        \end{array} 
\right) \leq 
\left(
        \begin{array}{cc}
        0        & a_n      \\
        a_n      & 0 
        \end{array} 
\right) \leq
\left(
        \begin{array}{cc}
        |a_n-1|  & 1        \\
        1        & |a_n-1| 
        \end{array} 
\right)
\end{equation} 
Thus the above bound implies by repeated use at each point of indices:
\begin{equation}\label{impbound}
W(\{a_n\equiv1\},\{b_n^-\})\leq W(\{a_n\},\{b_n\})\leq W(\{a_n\equiv1\},\{b_n^+\})
\end{equation}
\noindent where $b_n^{\pm}$ is given by 
$$
b_n^{\pm}=b_n\pm(|a_{n-1}-1|+|a_n-1|).
$$
Thus using \eqref{W=D1} and Theorem \ref{AisLieb} and the first inequality in \eqref{impbound} we have
\begin{equation}\label{Almost-final1}
\sum_{j=1}^{N_-}|E^{-}_j+2|^\gamma \leq
 \frac{\pi}{\sqrt 3}\,  L_{\gamma,1}^{cl}
 \sum_n  \Big((b_n)_{-} +  |a_{n-1}| +  |a_{n}|\Big)^{\gamma+1/2}.
\end{equation}
Similarly using \eqref{W=D2} and Corollary \ref{AisLieb+} and the second inequality in \eqref{impbound}
we find
\begin{equation}\label{Almost-final2}
\sum_{j=1}^{N_+}|E^{+}_j-2|^\gamma \leq
 \frac{\pi}{\sqrt 3}\, L_{\gamma,1}^{cl}
 \sum_n  \Big((b_n)_{+} +  |a_{n-1}| +  |a_{n}|\Big)^{\gamma+1/2}.
\end{equation}
Note that for any $q\geq1$,
\begin{equation}\label{Jensen}
(\alpha+\beta+\gamma)^q  =  3^q \Big(\frac{\alpha}{3}+\frac{\beta}{3}+\frac{\gamma}{3}\Big)^q
                        \leq 3^q (\alpha^q+\beta^q+\gamma^q).
\end{equation}
Applying \eqref{Jensen} to each of \eqref{Almost-final1} and \eqref{Almost-final2} and summing them up we finally arrive at
$$
\sum_{j=1}^{N_-} |E^{-}_{j}+2|^\gamma +
\sum_{j=1}^{N_+} |E^{+}_{j}-2|^\gamma\leq 3^{\gamma-1/2}\,\frac{\pi}{\sqrt{3}}\,L^{cl}_{\gamma,1} 
\left(\sum_n |b_n|^{\gamma+1/2}+4\sum_n|a_n-1|^{\gamma+1/2})\right).
$$
The proof of Theorem \ref{main} is complete.

\end{document}